\documentclass[12pt]{article}
\usepackage{amsmath,amssymb,amsthm,wasysym}
\usepackage{fullpage}
\newtheorem{theorem}{\textbf{Theorem}}
\newtheorem{proposition}{\textbf{Proposition}}
\newtheorem{conjecture}{\textbf{Conjecture}}
\newtheorem{example}{\textbf{Example}}
\newtheorem{lemma}{\textbf{Lemma}}

\begin{document}

\title{Casselman's Basis of Iwahori vectors and Kazhdan-Lusztig
  polynomials}
\author{Daniel Bump
  and Maki Nakasuji\footnote{This work was supported
by NSF grant DMS-1601026 and JSPS Grant-in-Aid for Young Scientists (B) 15K17519.}}
\maketitle

\begin{abstract}
  A problem in representation theory of $p$-adic groups
  is the computation of the \textit{Casselman basis} of
  Iwahori fixed vectors in the spherical principal series
  representations, which are dual to the intertwining
  integrals. We shall express the transition matrix
  $(m_{u,v})$ of the Casselman basis to another natural basis in
  terms of certain polynomials which are deformations
  of the Kazhdan-Lusztig R-polynomials. As an application
  we will obtain certain new functional equations
  for these transition matrices under the algebraic
  involution sending the residue cardinality $q$ to
  $q^{-1}$. We will also obtain a new proof of a
  surprising result of Nakasuji and Naruse that
  relates the matrix $(m_{u,v})$ to its inverse.
\end{abstract}

\section{Statement of Results}
We will state most of our results in this section, with proofs
in Section~\ref{proofssec}. A few more results will be stated
in Section~\ref{despro}.

Let $q$ be the residue cardinality of $F$ and $\mathfrak{o}$ its ring of
integers. Let $\hat{T} (\mathbb{C})$ be a split maximal torus in the
Langlands dual group $\hat{G} (\mathbb{C})$, a reductive algebraic group over
$\mathbb{C}$. Let $\Phi$ be the root system of $\hat{G}$ in the weight
lattice $X^{\ast} (\hat{T})$ of rational characters of $\hat{T}$ which we
identify with the group $X_{\ast} (T)$ of cocharacters in the maximal torus
$T$ of $G$ that is dual to $\hat{T}$. Let $B = T U$ be the Borel subgroup of
$G$ that is positive with respect to a decomposition of $\Phi$ into positive
and negative roots. Let $K$ be the standard (special) maximal compact
subgroup, and $J$ the positive Iwahori subgroup. The Weyl group $W = N_G (T
(F)) / T (F)$. We will choose Weyl group representatives from $K$.

If $\mathbf{z} \in \hat{T}$ then $\mathbf{z}$ parametrizes an unramified
character $\chi_{\mathbf{z}}$ of $T (F)$. The corresponding principal series
module $V_{\mathbf{z}}$ of $G (F)$ consists of smooth functions $f$ on $G
(F)$ such that $f (b g) = (\delta^{1 / 2} \chi_{\mathbf{z}}) (b) f (g)$ for
$b \in B (F)$. If $w \in W$, then choosing a Weyl group representative from
$K$, and by abuse of notation denoting it also as $w$, there is an
{\textit{intertwining integral operator}} $\mathcal{A}_w : V_{\mathbf{z}}
\longrightarrow V_{w\mathbf{z}}$ defined by the integral
\[ \mathcal{A}_w f (g) =_{} \int_{U \cap w U_- w^{- 1}} f (w^{-
   1} x) \, d x. \]
Here $U_-$ is the unipotent radical of the Borel subgroup $B_-$ opposite $B$,
and although the integral is only convergent for $\mathbf{z}$ in an open
subset of $\hat{T} (\mathbb{C})$, it extends meromorphically to all of
$\hat{T} (\mathbb{C})$ by analytic continuation. Casselman
{\cite{CasselmanSpherical}} and Casselman and
Shalika~{\cite{CasselmanShalika}} emphasized the importance of the functionals
$f \mapsto \mathcal{A}_w f (1)$ on the $| W |$-dimensional space
$V_{\mathbf{z}}^J$ of Iwahori fixed vectors.

The space $V_{\mathbf{z}}^J$ of Iwahori-fixed vectors in $V_{\mathbf{z}}$
then have several important bases parametrized by the Weyl group $W$.
One basis $\{\phi_w\}$ is obtained by restricting the standard spherical
vector to the various cells in the Bruhat decomposition. That is, $G (F)$ is
the disjoint union over $w \in W$ of cells $B w J$, so if $w \in W$ we may
define
\[ \phi_w (b k) = \left\{ \begin{array}{lll}
     (\delta^{1 / 2} \chi_{\mathbf{z}}) (b) & \text{if $k \in$} J w J, & \\
     0 & \text{otherwise} . & 
   \end{array} \right. \]
For us, a more useful basis is
\[ {\psi}_w = \sum_{u \geqslant w} \phi_u, \]
where $\geqslant$ is the Bruhat order in $W$.

Another more subtle basis than the $\{\phi_w\}$ or $\{\psi_w\}$
was defined in {\cite{CasselmanSpherical}} to be dual to the functionals
$f \mapsto \mathcal{A}_w f (1)$. Thus $\mathcal{A}_w f_{w'}(1)=\delta_{w,w'}$.
Casselman wrote:
\begin{quotation}It is an unsolved problem and, as far as I can see, a
    difficult one to express the bases $\{\phi_w\}$ and $\{f_w\}$ in terms of
    one another.
\end{quotation}
It seems more natural to ask for the transition function between
the bases $\{f_w\}$ and $\{\psi_w\}$, and we will interpret the
``Casselman problem'' to mean this question.

The difficulty of this problem did not prevent the use
of the Casselman basis $\{f_w\}$ in applications, for as
Casselman~\cite{CasselmanSpherical} and
Casselman-Shalika~\cite{CasselmanShalika} showed, a
small amount of information about the Casselman basis
can be used to compute special functions such as the
spherical and Whittaker functions. This is an idea that
has been used in a great deal of subsequent
literature. Because detailed information about
the Casselman basis is not needed for these
proofs, the Casselman problem has not seemed
urgent. Nevertheless, the Casselman problem is very
interesting in its own right because of a deep
underlying structure similar to Kazhdan-Lusztig
theory.

Before continuing, we remark that we will often
find functions $(u,v)\mapsto a_{u,v}$ on $W\times W$ such that
$a_{u,v}$ vanishes unless $u\leqslant v$. It is
convenient to think of $(a_{u,v})_{u,v\in W}$
as a matrix whose index set is the Weyl group.
Its product with another such matrix $(b_{u,v})$
is $(c_{u,v})$ where
\[c_{u,v}=\sum_{u\leqslant x\leqslant v} a_{u,x}b_{x,v}\,.\]
An important special case is the matrix $(a_{u,v})$
where $a_{u,v}=1$ if $u\leqslant v$ and $0$ otherwise. Then a theorem
of Verma which we will often use, is that if
$(b_{u,v})$ is the inverse matrix, then
$b_{u,v}=(-1)^{l(v)-l(u)}$ when $u\leqslant v$.
This the M\"obius function for the Bruhat order. See
\cite{Verma,Stembridge}.

Applying Casselman's functionals to the basis
$\{\psi_w\}$ give numbers
\[ m_{u, v} =\mathcal{A}_v {\psi}_u (1), \]
and these are the subject of this paper, as well as {\cite{BumpNakasuji}}.
This is zero unless $u \leqslant v$ in the Bruhat order.

We also
let $m'_{u,v}$ (denoted $\widetilde{m}_{u,v}$ in \cite{BumpNakasuji})
denote the inverse matrix so that
\[\sum_{u\leqslant x\leqslant v}m_{u,x}m'_{x,v}=\delta_{u,v},\]
where $\delta$ is the Kronecker delta. Clearly
\[\psi_{v}=\sum_{u\leqslant v} m_{u,v}\,f_u,\qquad
f_{v}=\sum_{u\leqslant v} m'_{u,v}\,\psi_u,\]
so the essence of the Casselman problem is to understand
the $m_{u,v}$ and~$m'_{u,v}$. We will give a kind of
solution to this problem by showing that the $m_{u,v}$
and $m'_{u,v}$ can expressed in terms of certain
polynomials which are deformations of the
Kazhdan-Lusztig R-polynomials.

First, review two conjectures from our previous
paper~\cite{BumpNakasuji}. Let $P_{u, v}$ be the
Kazhdan-Lusztig polynomials for $W$, defined as in
{\cite{KazhdanLusztig}}. We will also use the
{\textit{inverse Kazhdan-Lusztig polynomial}} $Q_{u, v}
= P_{w_0 v, w_0 u}$, where $w_0$ is the long Weyl group
element. Both $P_{u, v}$ and $Q_{u, v}$ vanish unless $u
\leqslant v$.

If $\alpha \in \Phi$, let $r_{\alpha}$ denote the corresponding reflection in
$W$. Assume that $u \leqslant v$. Define
\[ S (u, v) = \{ \alpha \in \Phi^+ |u \leqslant v.r_{\alpha} < v \},\qquad
S' (u, v) = \{ \alpha \in \Phi^+ |u \leqslant u.r_{\alpha} < v \}\,. \]
It is a consequence of work of Deodhar~\cite{DeodharConjecture},
Carrell and Peterson~\cite{Carrell}, Polo~\cite{Polo}, Dyer~\cite{Dyer} and
Jantzen~\cite{JantzenGewicht} that the sets $S(u,v)$ and
$S'(u,v)$ have cardinality $\geqslant l (v) - l (u)$. Moreover if
the inverse Kazhdan-Lusztig polynomial $Q_{u, v} = 1$,
then $| S (u, v) | = l (v) - l (u)$, while if $P_{u,v}=1$
then $| S' (u, v) | = l (v) - l (u)$.

In {\cite{BumpNakasuji}} we conjectured that if $\Phi$ is simply-laced
and $Q_{u, v} = 1$, then
\begin{equation}
  \label{conjecture} m_{u, v} = m_{u, v} (\mathbf{z}) = \prod_{\alpha \in S
  (u, v)} \frac{1 - q^{- 1} \mathbf{z}^{\alpha}}{1 -\mathbf{z}^{\alpha}} .
\end{equation}
This formula generalizes the well-known {\textit{formula of Gindikin and
Karpelevich}}, which is actually due to Langlands~{\cite{LanglandsEuler}} in
this nonarchimedean setting. This is the special case where $u = 1$, so that
${\psi}_1$ is the $K$-spherical vector in $V_{\mathbf{z}}$. However
the method commonly used to prove the formula of Gindikin and Karpelevich
inductively does not work for general $u$, and this conjecture still seems
difficult.  See {\cite{LeeLenartEtAl}} and {\cite{NakasujiNaruse}} for recent
work on this problem, and Section~\ref{despro} below for some new results
based on the methods of this paper.

Similarly, if $P_{u,v}=1$, then $|S'(u,v)|=l(v)-l(u)$, and in this
case we conjectured that
\begin{equation}
  \label{altconjecture}
m'_{u,v}=(-1)^{l(v)-l(u)}\,\prod_{\alpha \in S'
  (u, v)} \frac{1 - q^{- 1} \mathbf{z}^{\alpha}}{1 -\mathbf{z}^{\alpha}} .
\end{equation}
It was shown by Nakasuji and Naruse~\cite{NakasujiNaruse} that these
two conjectured formulas (\ref{conjecture}) and (\ref{altconjecture})
are equivalent. They did this by proving a very
interesting fact relating the matrices $(m_{u,v})$ and
$(m'_{u,v})$ which we will reprove in this paper as
Theorem~\ref{mmduality} below.

In this paper we will not prove these conjectures. Instead we
will strive to adapt methods of Kazhdan and
Lusztig~\cite{KazhdanLusztig} to this situation.
For example the above conjectures may be thought
of as closely related to their formula (2.6.b).

Our algebraic results about $m_{u, v}$ are independent of the origin of the
problem in $p$-adic groups. So we may regard $q$ as an indeterminate. If
$f$ is a polynomial in $q$, following Kazhdan and Lusztig, $\overline{f}$ will
denote the result of replacing $q$ by $q^{- 1}$. If $f$ involves
$\mathbf{z}$, then $\mathbf{z}$ is unchanged in $\overline{f}$ unless we
explicitly indicate a change. We will also the notations $\varepsilon_w = (-
1)^{l (w)}$ and $q_w = q^{l (w)}$ from~{\cite{KazhdanLusztig}}.

Assume that $Q_{u, v} = 1$, that $\Phi$ is simply-laced so that
(\ref{conjecture}) is conjectured, and moreover $| S (u, v) | = l (v) - l
(u)$. Observe that $m_{u, v}$ satisfies the functional equation
\begin{equation}
  \label{functionalequation} \overline{m_{u, v}} (\mathbf{z}) = q_v q_u^{-
  1} m_{u, v} (\mathbf{z}^{- 1}) .
\end{equation}

\begin{theorem}
  \label{qonefe}
  Assume that $Q_{u, v} = 1$. Then the functional equation
  (\ref{functionalequation}) is satisfied.
\end{theorem}

Note that this does not require $\Phi$ to be simply-laced, even though
(\ref{conjecture}) has counterexamples already for $B_2$.
Proofs will be in the next section.

The key to this and other results is to introduce a deformation of the Kazhdan
and Lusztig R-polynomials, defined in {\cite{KazhdanLusztig}}.

\begin{theorem}
  \label{rrecursion}
  There exist polynomials
  $r_{u, v} (\mathbf{z})$, depending on $\mathbf{z} \in \hat{T}
  (\mathbb{C})$ such that $r_{u,u}=1$ and $r_{u,v}=0$
  unless $u\leqslant v$. They have the property
  that $r_{u, v} (\mathbf{z}) \longrightarrow R_{u, v}$ if $\mathbf{z}
  \longrightarrow \infty$ in such a direction that $\mathbf{z}^{\alpha}
  \longrightarrow \infty$ for all positive roots $\alpha \in \Phi^+$. They may
  be calculated by the following recursion formula. Choose a simple
  reflection $s = s_{\alpha}$ corresponding to the simple root $\alpha$ such
  that $s v < v$.  If $s u < u$, then
  \[ r_{u, v} (\mathbf{z}) = \frac{1 - q}{1 - \mathbf{z}^{-v^{-1}\alpha}}
     r_{u, s v} (\mathbf{z}) + r_{s u, s v} (\mathbf{z}) . \]
  If $s u > u$, then
  \[ r_{u, v} (\mathbf{z}) = (1 - q) \frac{\mathbf{z}^{-v^{-1}\alpha}}{1 -
     \mathbf{z}^{-v^{-1}\alpha}} r_{u, s v} (\mathbf{z}) + q r_{s u, s v}
     (\mathbf{z}) . \]
\end{theorem}

\medbreak\noindent
In the recursion, it is worth noting that since $sv<v$, $-v^{-1}\alpha$ is a positive root.

Then the $m_{u,v}$ can be expressed in terms of the $r_{u,v}$ as follows.

\begin{theorem}
  \label{randmtheorem}
  Suppose that $u \leqslant v$. Then
  \begin{equation}
    \label{mfromr} m_{u, v} = \sum_{u \leqslant x \leqslant v} \overline{r_{x,
    v}}
  \end{equation}
  and
  \begin{equation}
    \label{rfromm} r_{u, v} = \sum_{u \leqslant x \leqslant v} \varepsilon_u
    \varepsilon_x \overline{m_{x, v}} .
  \end{equation}
\end{theorem}

Proof will be in the next section.
We will deduce (\ref{functionalequation}) from this result. Moreover, we will
prove the following general identity. If $u \leqslant v$ define
\begin{equation}
  \label{cuvdefinition}
  c_{u, v} = \sum_{u \leqslant x \leqslant y \leqslant z \leqslant v}
  \varepsilon_x \varepsilon_y q_y^{- 1} q_u P_{x, y} \overline{Q_{y, z}}
  \varepsilon_z \varepsilon_v .
\end{equation}
(Let $c_{u,v}=0$ if $u$ is not $\leqslant v$.)

\begin{theorem}
  \label{fullerinversionthm}
  If $u \leqslant v$ then
  \begin{equation}
    \label{fullerinversion} \overline{m_{u, v}} (\mathbf{z}) = q_v q_u^{- 1}
    \sum_{u \leqslant w \leqslant v} c_{u, w} m_{w, v} (\mathbf{z}^{- 1}) .
  \end{equation}
\end{theorem}

Proof will be in the next section.
The coefficients $c_{u, v}$ are interesting. If $u=v$ then $c_{u,v}=1$,
but otherwise they are usually zero. The 46 pairs $u,v$ with $c_{u,v}\neq 1$
and $u<v$ for the $A_4$ Weyl group are tabulated in Figure~\ref{a4case}. This includes
all 38 pairs of Weyl group elements with $u\prec v$ in the notation
of Kazhdan and Lusztig. This means that $l (v) - l (u)$ is odd
and $\geqslant 3$, and that the degree of $P_{u, v}$ is $\frac{1}{2} (l (v) -
l (u) - 1)$, the largest possible. But there are a few other values for which
$c_{u, v} \neq 0$.

\begin{figure}[h]
\[\begin{array}{|l|l|c|c||l|l|c|c|}
\hline
\qquad u & \quad v & c_{u,v} & u\prec v &&&&\\
\hline
\scriptstyle{s_{3}s_{2}} & \scriptstyle{s_{3}s_{4}s_{2}s_{3}s_{1}s_{2}} &\scriptstyle{ q^{-1} - q^{-3} } & &
\scriptstyle{s_{4}s_{1}s_{2}} & \scriptstyle{s_{1}s_{2}s_{3}s_{4}s_{3}s_{2}} &\scriptstyle{ q^{-1} - q^{-2} } & \checked \\
\hline
\scriptstyle{s_{3}s_{1}} & \scriptstyle{s_{3}s_{4}s_{2}s_{3}s_{1}} &\scriptstyle{ q^{-1} - q^{-2} } & \checked &
\scriptstyle{s_{4}s_{2}} & \scriptstyle{s_{4}s_{2}s_{3}s_{1}s_{2}} &\scriptstyle{ q^{-1} - q^{-2} } & \checked \\
\hline
\scriptstyle{s_{4}s_{1}s_{2}s_{1}} & \scriptstyle{s_{1}s_{2}s_{3}s_{4}s_{3}s_{2}s_{1}} &\scriptstyle{ q^{-1} - q^{-2} } & \checked &
\scriptstyle{s_{3}s_{1}s_{2}} & \scriptstyle{s_{3}s_{4}s_{2}s_{3}s_{1}s_{2}} &\scriptstyle{ q^{-1} - q^{-2} } & \checked \\
\hline
\scriptstyle{s_{2}s_{3}s_{2}} & \scriptstyle{s_{2}s_{3}s_{4}s_{1}s_{2}s_{3}} &\scriptstyle{ q^{-1} - q^{-2} } & \checked &
\scriptstyle{s_{1}s_{2}s_{3}s_{2}} & \scriptstyle{s_{3}s_{4}s_{1}s_{2}s_{3}s_{1}s_{2}} &\scriptstyle{ q^{-1} - q^{-2} } & \checked \\
\hline
\scriptstyle{s_{4}s_{2}} & \scriptstyle{s_{2}s_{3}s_{4}s_{3}s_{1}s_{2}} &\scriptstyle{ -q^{-1} + q^{-3} } &  &
\scriptstyle{s_{2}s_{3}s_{2}} & \scriptstyle{s_{2}s_{3}s_{4}s_{1}s_{2}s_{3}s_{1}s_{2}} &\scriptstyle{ q^{-2} - q^{-3} } & \checked \\
\hline
\scriptstyle{s_{3}s_{4}s_{1}s_{2}s_{1}} & \scriptstyle{s_{1}s_{2}s_{3}s_{4}s_{2}s_{3}s_{1}s_{2}} &\scriptstyle{ q^{-1} - q^{-2} } & \checked &
\scriptstyle{s_{3}s_{4}s_{3}s_{1}} & \scriptstyle{s_{1}s_{2}s_{3}s_{4}s_{2}s_{3}s_{2}s_{1}} &\scriptstyle{ -q^{-1} + q^{-3} } &  \\
\hline
\scriptstyle{s_{4}s_{2}} & \scriptstyle{s_{2}s_{3}s_{4}s_{3}s_{2}} &\scriptstyle{ q^{-1} - q^{-2} } & \checked &
\scriptstyle{s_{3}s_{4}s_{3}s_{1}s_{2}s_{1}} & \scriptstyle{s_{1}s_{2}s_{3}s_{4}s_{2}s_{3}s_{1}s_{2}s_{1}} &\scriptstyle{ q^{-1} - q^{-2} } & \checked \\
\hline
\scriptstyle{s_{3}s_{1}} & \scriptstyle{s_{1}s_{2}s_{3}s_{2}s_{1}} &\scriptstyle{ q^{-1} - q^{-2} } & \checked &
\scriptstyle{s_{4}s_{2}s_{3}s_{1}} & \scriptstyle{s_{2}s_{3}s_{4}s_{1}s_{2}s_{3}s_{2}s_{1}} &\scriptstyle{ q^{-1} - q^{-3} } &  \\
\hline
\scriptstyle{s_{3}s_{4}s_{1}} & \scriptstyle{s_{1}s_{2}s_{3}s_{4}s_{2}s_{1}} &\scriptstyle{ q^{-1} - q^{-2} } & \checked &
\scriptstyle{s_{2}} & \scriptstyle{s_{2}s_{3}s_{1}s_{2}} &\scriptstyle{ q^{-1} - q^{-2} } & \checked \\
\hline
\scriptstyle{s_{2}s_{3}s_{4}s_{2}} & \scriptstyle{s_{2}s_{3}s_{4}s_{2}s_{3}s_{1}s_{2}} &\scriptstyle{ q^{-1} - q^{-2} } & \checked &
\scriptstyle{s_{4}s_{1}s_{2}s_{1}} & \scriptstyle{s_{1}s_{2}s_{3}s_{4}s_{3}s_{1}s_{2}s_{1}} &\scriptstyle{ -q^{-1} + q^{-3} } &  \\
\hline
\scriptstyle{s_{2}s_{3}s_{2}} & \scriptstyle{s_{3}s_{4}s_{2}s_{3}s_{1}s_{2}} &\scriptstyle{ q^{-1} - q^{-2} } & \checked &
\scriptstyle{s_{4}s_{2}s_{3}s_{2}} & \scriptstyle{s_{2}s_{3}s_{4}s_{1}s_{2}s_{3}s_{2}} &\scriptstyle{ q^{-1} - q^{-2} } & \checked \\
\hline
\scriptstyle{s_{4}s_{1}} & \scriptstyle{s_{1}s_{2}s_{3}s_{4}s_{3}s_{2}s_{1}} &\scriptstyle{ q^{-2} - q^{-3} } & \checked &
\scriptstyle{s_{4}s_{2}s_{3}} & \scriptstyle{s_{2}s_{3}s_{4}s_{1}s_{2}s_{3}} &\scriptstyle{ q^{-1} - q^{-2} } & \checked \\
\hline
\scriptstyle{s_{3}s_{1}} & \scriptstyle{s_{3}s_{4}s_{1}s_{2}s_{3}} &\scriptstyle{ q^{-1} - q^{-2} } & \checked &
\scriptstyle{s_{2}s_{3}} & \scriptstyle{s_{2}s_{3}s_{4}s_{1}s_{2}s_{3}} &\scriptstyle{ q^{-1} - q^{-3} } &  \\
\hline
\scriptstyle{s_{4}s_{2}s_{3}s_{2}s_{1}} & \scriptstyle{s_{2}s_{3}s_{4}s_{1}s_{2}s_{3}s_{2}s_{1}} &\scriptstyle{ q^{-1} - q^{-2} } & \checked &
\scriptstyle{s_{4}s_{3}s_{1}} & \scriptstyle{s_{4}s_{1}s_{2}s_{3}s_{2}s_{1}} &\scriptstyle{ q^{-1} - q^{-2} } & \checked \\
\hline
\scriptstyle{s_{1}s_{2}s_{3}s_{4}s_{3}s_{1}} & \scriptstyle{s_{1}s_{2}s_{3}s_{4}s_{1}s_{2}s_{3}s_{2}s_{1}} &\scriptstyle{ q^{-1} - q^{-2} } & \checked &
\scriptstyle{s_{3}s_{4}s_{1}s_{2}} & \scriptstyle{s_{1}s_{2}s_{3}s_{4}s_{2}s_{3}s_{1}s_{2}} &\scriptstyle{ q^{-1} - q^{-3} } &  \\
\hline
\scriptstyle{s_{4}s_{1}s_{2}s_{1}} & \scriptstyle{s_{1}s_{2}s_{3}s_{4}s_{3}s_{1}s_{2}} &\scriptstyle{ q^{-1} - q^{-2} } & \checked &
\scriptstyle{s_{3}s_{4}s_{3}s_{1}} & \scriptstyle{s_{3}s_{4}s_{1}s_{2}s_{3}s_{2}s_{1}} &\scriptstyle{ q^{-1} - q^{-2} } & \checked \\
\hline
\scriptstyle{s_{4}s_{2}} & \scriptstyle{s_{2}s_{3}s_{4}s_{1}s_{2}} &\scriptstyle{ q^{-1} - q^{-2} } & \checked &
\scriptstyle{s_{3}s_{4}s_{3}s_{1}} & \scriptstyle{s_{1}s_{2}s_{3}s_{4}s_{2}s_{3}s_{1}} &\scriptstyle{ q^{-1} - q^{-2} } & \checked \\
\hline
\scriptstyle{s_{2}s_{3}s_{4}s_{3}s_{1}} & \scriptstyle{s_{2}s_{3}s_{4}s_{1}s_{2}s_{3}s_{2}s_{1}} &\scriptstyle{ q^{-1} - q^{-2} } & \checked &
\scriptstyle{s_{4}s_{1}s_{2}s_{3}s_{1}} & \scriptstyle{s_{2}s_{3}s_{4}s_{1}s_{2}s_{3}s_{2}s_{1}} &\scriptstyle{ q^{-1} - q^{-2} } & \checked \\
\hline
\scriptstyle{s_{2}s_{3}s_{2}s_{1}} & \scriptstyle{s_{2}s_{3}s_{4}s_{1}s_{2}s_{3}s_{1}} &\scriptstyle{ q^{-1} - q^{-2} } & \checked &
\scriptstyle{s_{3}s_{1}} & \scriptstyle{s_{3}s_{4}s_{1}s_{2}s_{3}s_{1}} &\scriptstyle{ -q^{-1} + q^{-3} } &  \\
\hline
\scriptstyle{s_{3}s_{4}s_{3}s_{1}s_{2}} & \scriptstyle{s_{1}s_{2}s_{3}s_{4}s_{2}s_{3}s_{1}s_{2}} &\scriptstyle{ q^{-1} - q^{-2} } & \checked &
\scriptstyle{s_{2}s_{3}s_{1}} & \scriptstyle{s_{2}s_{3}s_{4}s_{1}s_{2}s_{3}} &\scriptstyle{ q^{-1} - q^{-2} } & \checked \\
\hline
\scriptstyle{s_{4}s_{1}s_{2}s_{1}} & \scriptstyle{s_{2}s_{3}s_{4}s_{3}s_{1}s_{2}s_{1}} &\scriptstyle{ q^{-1} - q^{-2} } & \checked &
\scriptstyle{s_{4}s_{2}s_{1}} & \scriptstyle{s_{2}s_{3}s_{4}s_{3}s_{2}s_{1}} &\scriptstyle{ q^{-1} - q^{-2} } & \checked \\
\hline
\scriptstyle{s_{3}s_{4}s_{3}s_{1}} & \scriptstyle{s_{1}s_{2}s_{3}s_{4}s_{3}s_{2}s_{1}} &\scriptstyle{ q^{-1} - q^{-2} } & \checked &
\scriptstyle{s_{3}} & \scriptstyle{s_{3}s_{4}s_{2}s_{3}} &\scriptstyle{ q^{-1} - q^{-2} } & \checked \\
\hline
\scriptstyle{s_{3}s_{4}s_{2}} & \scriptstyle{s_{3}s_{4}s_{2}s_{3}s_{1}s_{2}} &\scriptstyle{ q^{-1} - q^{-2} } & \checked &
\scriptstyle{s_{1}s_{2}s_{3}s_{4}s_{2}} &
\scriptstyle{s_{1}s_{2}s_{3}s_{4}s_{2}s_{3}s_{1}s_{2}}a &\scriptstyle{ q^{-1} - q^{-2} } & \checked \\
\hline
\end{array}
\]
\caption{The pairs $u$,$v$ in the $A_4$ Weyl group with $u<v$ and $c_{u,v}\ne 0$. The simple reflections
are $s_1$, $s_2$, $s_3$ and $s_4$. This list includes all 38 pairs with $u\prec v$ in the notation of Kazhdan and Lusztig (marked with $\checked$). Note that if $u\prec v$ then $c_{u,v}=q^{-1}-q^{-2}$ but there are a few other pairs $u,v$ with $c_{u,v}\neq 0$.}
\label{a4case}
\end{figure}

Finally we have a striking symmetry of the coefficients
$m_{u,v}$. Equation (\ref{second_mm}) in the following
theorem was proved previously by Nakasuji and
Naruse~\cite{NakasujiNaruse}. We will give another proof based on
Theorem~\ref{rrecursion}.

\begin{theorem}
  \label{mmduality}
  (Nakasuji and Naruse~\cite{NakasujiNaruse}.)
  Suppose that $u \leqslant v$. Then
  \begin{equation}
    \label{first_mm}
    \sum_{u \leqslant x \leqslant v} r_{u, x} \varepsilon_x \varepsilon_v
    r_{w_0 v, w_0 x} = \delta_{u, v},
  \end{equation}
  and
  \begin{equation}
    \label{second_mm}
    \sum_{u \leqslant x \leqslant v} m_{u, x} \varepsilon_x \varepsilon_v
     m_{w_0 v, w_0 x} = \delta_{u, v}.
  \end{equation}
\end{theorem}

Proof will be in the next section.
Because $(m_{u, v}')$ was defined to be the inverse of the matrix $(m_{u,
v})$, the last result can be written $m'_{u, v} = \varepsilon_u \varepsilon_v
m_{w_0 v, w_0 u}$. This seems a remarkable fact.

We end this section with a conjecture about the poles of $m_{u,v}$.
As functions of $\mathbf{z}$, the function $r_{u, v} (\mathbf{z})$ is
analytic on the regular set of $\hat{T}$, that is, the subset of
$\mathbf{z}$ such that $\mathbf{z}^{\alpha} \neq 1$ for all $\alpha \in
\Phi$.

\begin{conjecture}
  The functions
  \[ \prod_{\beta \in S (u, v)} (1 -\mathbf{z}^{\beta}) m_{u, v}, \qquad
     \prod_{\beta \in S (u, v)} (1 -\mathbf{z}^{\beta}) r_{u, v} \]
  are analytic on all of $\hat{T} (\mathbb{C})$.
\end{conjecture}

Since $m_{u, v} = \sum_{u \leqslant x \leqslant v} \overline{r_{x, v}}$ and $S
(x, v) \subseteq S (u, v)$ when $u \leqslant x \leqslant v$, the statement
about $m_{u, v}$ follows from the statement about $r_{u, v}$. Moreover, the
recursion in Theorem~\ref{rrecursion} gives a way of trying to prove this
recursively. So let us choose a simple reflection $s$ such that $s v < v$. It
is sufficient to show that $\prod_{\alpha \in S (u, v)} (1
-\mathbf{z}^{\beta})$ cancels the poles of both $r_{s u, s v}$ and of $(1
-\mathbf{z}^{- v^{- 1} \alpha})^{- 1} r_{u, s v}$.

The factor $(1 -\mathbf{z}^{- v^{- 1} \alpha})^{- 1}$ that appears with
$r_{u, s v}$ is cancelled for the following reason. It only appears if $r_{u,
s v} \neq 0$, that is, if $u \leqslant s v$. Now if this is so, then the
positive root $- v^{- 1} \alpha$ is in $S (u, v)$, because $v r_{- v^{- 1}
\alpha} = s v$ and then $u \leqslant s v$ implies $- v^{- 1} \alpha \in S (u,
v)$.

So the statement that $\prod_{\beta \in S (u, v)} (1 -\mathbf{z}^{\beta})$
cancels the poles of $r_{u, v}$ would follow recursively if we knew that $S
(u, s v)$ and $S (s u, s v)$ are both contained in $S (u, v)$. Unfortunately
this is not always true, as the following example shows.

\begin{example}
  \label{polex}Let $\Phi$ be the $A_2$ root system, with simple roots
  $\alpha_1$, $\alpha_2$ and corresponding simple reflections $s_1, s_2$. Let
  $u = s_1$, $v = s_1 s_2 s_1$ and $\beta = \alpha_1 + \alpha_2$. Then if we
  take $s = s_1$ we have $\beta \in S (u, s v)$ and $\beta \in S (s u, s v)$
  but $\beta \notin S (u, v)$. This means that the locus of
  $\mathbf{z}^{\beta} = 1$ is a pole of both terms in the recursion, but
  these poles cancel and it is not a pole of $r_{u, v} (\mathbf{z})$.
\end{example}

At the moment we do not have a proof that such cancellation always occurs, but
often it can be proved using a different descent. In Example~\ref{polex} with
$u, v$ and $\beta$ as given, we could instead take $s = s_2$, and then we find
that $\beta \notin S (u, s v)$ and $\beta \notin S (s u, s v)$, so $1
-\mathbf{z}^{\beta}$ does not divide the denominator of $r_{u, v}$.

\section{\label{proofssec}Proofs}
Let $\mathcal{H}$ be the Iwahori Hecke algebra of the Coxeter group $W$,
with basis elements $T_w$ for $w\in W$, such that $T_wT_{w'}=T_{ww'}$
if $l(ww')=l(w)l(w')$. Thus if $s$ is a simple reflection we
have $T_s^2=(q-1)T_s+q$, and the usual braid relations are satisfied.
We extend the scalars to the field of meromorphic functions on
$\widehat{T}(\mathbb{C})$. Then the Hecke algebra has another
basis which we will now describe. Let $\mathbf{z} \in \hat{T}
(\mathbb{C})$. If $s = s_{\alpha}$ is a simple reflection, and $\alpha$ is the
corresponding simple root, let ${\mu}_{\mathbf{z}} (s)$ be the element of the
Hecke algebra defined by
\[ {\mu}_{\mathbf{z}} (s) = q^{- 1} T_s + (1 - q^{- 1})
   \frac{\mathbf{z}^{\alpha}}{1 -\mathbf{z}^{\alpha}} = T_s^{- 1} +
   \frac{1 - q^{- 1}}{1 -\mathbf{z}^{\alpha}} . \]
It is shown in {\cite{BumpNakasuji}}, using ideas of Rogawski~\cite{Rogawski}
that we may extend this definition to
${\mu}_{\mathbf{z}} (w)$ for $w \in W$ such that if $l (w_1 w_2) = l
(w_1) + l (w_2)$ then
\[ {\mu}_{\mathbf{z}} (w_1 w_2) ={\mu}_{\mathbf{z}} (w_2)
   {\mu}_{w_2 \mathbf{z}} (w_1) . \]
The Hecke operator $\mu_w(\mathbf{z})$ models the intertwining
operator $\mathcal{A}_w:V_\mathbf{z}\to V_{w\mathbf{z}}$ as
is explained in \cite{Rogawski} or \cite{BumpNakasuji}.
It was clarified by Nakasuji and Naruse \cite{NakasujiNaruse} that the basis
$\mu_w$ is essentially the ``Yang-Baxter basis'' of
Lascoux, Leclerc and Thibon~\cite{LLT}, and the consistency
of the definition follows from the Yang-Baxter equation.
The appearance of the Yang-Baxter equation in the
context of $p$-adic intertwining operators is then related
to the viewpoint in Brubaker, Buciumas, Bump and Friedberg~\cite{BBBF}.

Suppose
that $s = s_{\alpha}$ is a simple reflection. Then it is easy to check by
direct computation that
\begin{equation}
  \label{mumuconstant} \mu_{\mathbf{z}} (s) \mu_{s\mathbf{z}} (s) =
  \frac{1 - q^{- 1} \mathbf{z}^{\alpha}}{1 -\mathbf{z}^{\alpha}} \cdot
  \frac{1 - q^{- 1} \mathbf{z}^{- \alpha}}{1 -\mathbf{z}^{- \alpha}} .
\end{equation}

\begin{lemma}
  \label{mumuela}Let $s = s_{\alpha}$ be a simple reflection. Then for any $w
  \in W$ we have $\mu_{\mathbf{z}} (w) \mu_{w\mathbf{z}} (s) = c \cdot
  \mu_{\mathbf{z}} (s w)$ where the constant
  \begin{equation}
    c = \left\{ \begin{array}{ll}
      1 & \text{if $s w > w$,}\\
      \frac{1 - q^{- 1} \mathbf{z}^{w^{- 1} \alpha}}{1 -\mathbf{z}^{w^{-
      1} \alpha}} \cdot \frac{1 - q^{- 1} \mathbf{z}^{- w^{- 1} \alpha}}{1
      -\mathbf{z}^{- w^{- 1} \alpha}} & \text{if $s w < w$.}
    \end{array} \right.
  \end{equation}
\end{lemma}

\begin{proof}
  If $s w > w$ this follows from the definition of $\mu_{\mathbf{z}} (s w)$.
  In the other case, we write $\mu_{\mathbf{z}} (w) = \mu_{\mathbf{z}} (s
  w) \mu_{s w\mathbf{z}} (s)$, then apply (\ref{mumuconstant}).
\end{proof}

Let $\Lambda :
\mathcal{H} \longrightarrow \mathbb{C} (q)$ be the functional such that
$\Lambda (T_w) = 1$ if $w = 1$, and $0$
otherwise. Also, let $\psi_w =\sum_{u\geqslant w} T_w$.
We are reusing the notation $\psi_w$ used previously to denote
certain Iwahori fixed vectors, but we are leaving the
origins of the problem in the $p$-adic group behind, so
this reuse should not cause any confusion. Following
Rogawski~\cite{Rogawski}, there is a vector space
isomorphism between the Iwahori fixed vectors and the
Hecke algebra $\mathcal{H}$, and in this isomorphism,
the Iwahori fixed vectors $\psi_w$ correspond to the Hecke
elements~$\psi_w$.

In \cite{BumpNakasuji} we prove
\begin{equation}
  m_{u, v} = m_{u, v} (\mathbf{z}) = \Lambda (\psi_u
  {\mu}_{\mathbf{z}} (v)) .
\end{equation}
This will be the starting point of our proofs.

\begin{lemma}
  \label{omegadelta}If $u, v \in W$ then
  \begin{equation}
    \label{omegakd} \Lambda (T_u T_v) = \left\{ \begin{array}{ll}
      q_u  & \text{if $u = v^{- 1}$,}\\
      0 & \text{otherwise.}
    \end{array} \right.
  \end{equation}
\end{lemma}

\begin{proof}
  Without loss of generality $l (u) \leqslant l (v)$. Assume that $\Lambda (T_u
  T_v) \neq 0$. We will show that $u = v^{- 1}$ and that $\Lambda (T_u T_v) =
  q_u$. Proof is by induction on $l (u)$, so we assume $\Lambda (T_{u'} T_v)$
  is given by this formula for all $u' < u$ and for all $v$. The formula
  (\ref{omegakd}) is trivial if $u = 1$, so we may assume $u > 1$. Let $s$ be
  a simple reflection such that $u s < u$. Let $u' = u s$ and $v' = s v$.
  
  Suppose that $v' < v$. Then $T_v = T_s T_{v'}$ and $T_u = T_{u'} T_s$. Thus
  \begin{equation}
    \label{tutvx} T_u T_v = T_{u'} T_s^2 T_{v'} = (q - 1) T_{u'} T_s T_{v'} +
    q T_{u'} T_{v'} = (q - 1) T_{u'} T_v + q T_{u'} T_{v'}\,.
  \end{equation}
  Thus either $\Lambda (T_{u'} T_v) \neq 0$ or $\Lambda (T_{u'} T_{v'}) \neq 0$.
  By induction we have either $u' = v^{- 1}$ or $u' = (v')^{- 1}$. The first
  is not possible since $l (u') < l (v)$, so $u' = (v')^{- 1}$ and $u^{- 1} =
  v^{- 1}$. Now applying $\Lambda$ to (\ref{tutvx}) gives $\Lambda (T_u T_v) = q
  \Lambda (T_{u'} T_{v'}) = q q_{u'} = q_u$.
  
  The case $v' > v$ is easier. Then $\Lambda (T_u T_v) = \Lambda (T_{u'}
  T_{s_i} T_v) = \Lambda (T_{u'} T_{v'}) = 0$ since $l (u') < l (v')$. And $u$
  cannot equal $v^{- 1}$ since $s$ is a right descent of $u$ but not $v^{-
  1}$.
\end{proof}

We will make use of the Kazhdan-Lusztig involution $f \mapsto \overline{f}$ on
functions $f$ of $q$, $\mathbf{z}$. This is the map that sends $q \mapsto
q^{- 1}$ and $\mathbf{z} \mapsto \mathbf{z}$. We recall from
{\cite{KazhdanLusztig}} that is the map that sends $q \mapsto q^{- 1}$, and it
is extended to an automorphism the Hecke algebra by the map $T_w \mapsto
T_{w^{- 1}}^{- 1}$.

We define $r_{u, v} = r_{u, v} (\mathbf{z})$ by
\begin{equation}
  \label{ruvdef}
{\mu}_{\mathbf{z}} (v) = \sum_{u \leqslant v} q_u^{- 1}
\overline{r_{u, v}} \, T_{u^{- 1}} .
\end{equation}
We will prove Theorem~\ref{rrecursion} before Theorem~\ref{randmtheorem},
and Theorem~\ref{qonefe} last.

\begin{proof}[Proof of Theorem~\ref{rrecursion}]
  Beginning with (\ref{ruvdef}), 
  we may compute $\overline{r_{u, v}}$ by calculating the coefficient of
  $T_{u^{- 1}}$ in
  \[ {\mu} (v) ={\mu}_{\mathbf{z}} (s v) {\mu}_{s v\mathbf{z}}
     (s) = \left( \sum_{x \leqslant s v} q_x^{- 1} \overline{r_{x, s v}}
     T_{x^{- 1}} \right) \left( q^{- 1} T_s + (1 - q^{- 1}) \frac{(s
     v\mathbf{z})^{\alpha_i}}{1 - (s v\mathbf{z})^{\alpha}} \right) . \]
  Only $x = u$ or $s u$ can contribute to the coefficient of $T_{u^{- 1}}$.
  Comparing the coefficients of $T_{u^{-1}}$ and noting that
  $(s v\mathbf{z})^{\alpha}=\mathbf{z}^{-v^{-1}\alpha}$, the recursion
  formula is obtained.

  Now the Kazhdan-Lusztig R-polynomials satisfy a similar
  recurrence, at the beginning of Section~2 in~\cite{KazhdanLusztig}.
  So specializing $\mathbf{z}\to\infty$ in such a way that
  $\mathbf{z}^\alpha\to\infty$ for all positive roots, we
  see that $r_{u,v}\to R_{u,v}$. For this it is important
  that when $s$ is a left descent of $v$, the root $-v^{-1}\alpha$
  that appears in the recursion is positive.
\end{proof}

\begin{proposition}
  We have $r_{u, v} = 0$ unless $u \leqslant v$, and $r_{v, v} = 1$. Moreover
  $\overline{r_{u, v}} = \varepsilon_u \varepsilon_v q_u q_v^{- 1} r_{u, v}$.
\end{proposition}

\begin{proof}
  Both assertions follow from Theorem~\ref{rrecursion} by induction on $l
  (v)$.
\end{proof}

If $u \leqslant v$ in $W$ we will denote by $[u, v]$ the Bruhat interval $\{ x
\in W|u \leqslant x \leqslant v \}$.

\begin{proof}[Proof of Theorem~\ref{randmtheorem}]
  By definition
  \[ m_{u, v} = \Lambda (\psi_u {\mu}_{\mathbf{z}} (v)) = \sum_{x
     \geqslant u} \sum_{y \leqslant v} q_y^{- 1} \overline{r_{y, v}} \Lambda
     (T_x T_{y^{- 1}}) . \]
  Equation (\ref{mfromr}) now follows from Lemma~\ref{omegadelta}. By Verma's
  theorem the M\"obius function on the Bruhat interval $[u, v]$ is $(x, y)
  \mapsto \varepsilon_x \varepsilon_y$. (See {\cite{Stembridge}}.) Thus
  (\ref{rfromm}) follows from (\ref{mfromr}).
\end{proof}

\begin{lemma}
  \label{muzinvlemma}We have
  \[ \overline{{\mu}_{\mathbf{z}} (w)} = q_w {\mu}_{\mathbf{z}^{-
     1}} (w) . \]
\end{lemma}

\begin{proof}
  This reduces to the case where $w$ is a simple reflection, and this case is
  easily checked from the definition.
\end{proof}

\begin{proposition}
  We have
  \begin{equation}
    \label{qmrelation} \sum_{u \leqslant w \leqslant t \leqslant v}
    \overline{Q_{u, w}} \varepsilon_w \varepsilon_t m_{t, v} (\mathbf{z}^{-
    1}) = q_u q_v^{- 1} \sum_{u \leqslant y \leqslant v} Q_{u, y} \overline{}
    r_{y, v} (\mathbf{z})
  \end{equation}
\end{proposition}

\begin{proof}
  Using Lemma~\ref{muzinvlemma},
  \[ q_v \sum_{u \leqslant v} q_u^{- 1} \overline{r_{u, v}} (\mathbf{z}^{-
     1}) T_{u^{- 1}} = q_v {\mu}_{\mathbf{z}^{- 1}} (v) =
     \overline{{\mu}_{\mathbf{z}} (v)} = \sum_{y \leqslant v} q_y r_{y,
     v} (\mathbf{z}) T_y^{- 1} . \]
  Now $T_y^{- 1} = \sum_{u \leqslant v} \overline{R_{u, y}} q_u^{- 1} T_{u^{-
  1}}$ by {\cite{KazhdanLusztig}} (2.0.a). Substituting this on the right-hand
  side, and comparing the coefficients of $T_{u^{- 1}}$ gives
  \begin{equation}
    \hspace{1em} \overline{r_{u, v}} (\mathbf{z}^{- 1}) = q_v^{- 1} \sum_{u
    \leqslant y \leqslant v} q_y \overline{R_{u, y}} r_{y, v} (\mathbf{z}) .
  \end{equation}
  Then by Theorem~\ref{randmtheorem},
  \[ m_{u, v} (\mathbf{z}^{- 1}) = \sum_{u \leqslant x \leqslant v}
     \overline{r_{x, v}} (\mathbf{z}^{- 1}) = \sum_{u \leqslant x \leqslant
     y \leqslant v} q_v^{- 1} q_y \overline{R_{x, y}} r_{y, v} (\mathbf{z})
     . \]
  By Verma's theorem ({\cite{Verma}}, {\cite{Stembridge}})
  \[ \sum_{u \leqslant t \leqslant v} \varepsilon_u \varepsilon_t m_{t, v}
     (\mathbf{z}^{- 1}) = \sum_{u \leqslant t \leqslant x \leqslant y
     \leqslant v} \varepsilon_u \varepsilon_t q_v^{- 1} q_y \overline{R_{x,
     y}} r_{y, v} (\mathbf{z}) = \sum_{u \leqslant y \leqslant v} q_v^{- 1}
     q_y \overline{R_{u, y}} r_{y, v} (\mathbf{z}) . \]
  Thus
  \[ \sum_{u \leqslant w \leqslant t \leqslant v} \overline{Q_{u, w}}
     \varepsilon_w \varepsilon_t m_{t, v} (\mathbf{z}^{- 1}) = \sum_{u
     \leqslant w \leqslant y \leqslant v} \overline{Q_{u, w}} q_v^{- 1} q_y
     \overline{R_{w, y}} r_{y, v} (\mathbf{z}) . \]
  \[ \  \]
  Now we require the identity
  \[ Q_{u, y} = q_u^{- 1} q_y \sum_{u \leqslant w \leqslant y} \overline{Q_{u,
     w}} \cdot \overline{R_{w, y}}, \]
  which may be deduced from {\cite{KazhdanLusztig}} (2.2.a). Applying this
  gives (\ref{qmrelation}).
\end{proof}

\begin{proof}[Proof of Theorem~\ref{qonefe}]
  Since $Q_{u, v} = 1$ we have $Q_{u, x} = 1$ for all $x \in [u, v]$. Thus
  (\ref{qmrelation}) reads:
  \[ \sum_{u \leqslant s \leqslant t \leqslant v} \varepsilon_s \varepsilon_t
     m_{t, v} (\mathbf{z}^{- 1}) = q_u q_v^{- 1} \sum_{u \leqslant y
     \leqslant v} \overline{} r_{y, v} (\mathbf{z}) = q_u q_v^{- 1}
     \overline{m_{u, v}} . \]
  The result follows from Verma's theorem.
\end{proof}

Let $c_{u,v}$ be as in (\ref{cuvdefinition}).

\begin{proof}[Proof of Theorem~\ref{fullerinversionthm}]
  Using (\ref{qmrelation}), write
  \[ q_v \sum_{x \leqslant y \leqslant z \leqslant t \leqslant v}
     \varepsilon_x \varepsilon_y q_y^{- 1} P_{x, y} \overline{Q_{y, z}}
     \varepsilon_z \varepsilon_t m_{t, v} (\mathbf{z}^{- 1}) = \sum_{x
     \leqslant y \leqslant t \leqslant v} \varepsilon_x \varepsilon_y P_{x, y}
     Q_{y, t} \overline{} r_{t, v} (\mathbf{z}) . \]
  Using the inversion formula Theorem~3.1 of {\cite{KazhdanLusztig}}, the
  right-hand side is just $r_{x, v} (\mathbf{z})$. Now summing over $x$ in
  $[u, v]$ and using (\ref{mfromr}) gives (\ref{fullerinversion}).
\end{proof}

In preparation for proving Theorem~\ref{mmduality},
define $r'_{u, v}$ to be the inverse of $(r_{u, v})$ regarded as a matrix on
$| W |$. Thus
\[ \sum_{u \leqslant x \leqslant v} r_{u, x} r'_{x, v} = \sum_{u \leqslant x
   \leqslant v} r_{u, x}' r_{x, v} = \delta_{u, v} . \]
Then using Verma's theorem it is easy to see that
\[ m_{u, v}' = \sum_{u \leqslant x \leqslant v} \varepsilon_x \varepsilon_v
   \overline{r'_{u, x}}, \qquad r_{u, v}' = \sum_{u \leqslant x \leqslant v}
   \overline{m'_{u, x}} . \]
The coefficient $r'_{u,v}(\mathbf{z})$ specializes to
$\varepsilon_{u}\,\varepsilon_{v}\,R(u,v)$ as
$\mathbf{z}\to\infty$. This is clear from
\cite{KazhdanLusztig}~Lemma~2.1(ii). Nevertheless, we are not aware
of any simple relationship between the coefficients $r$
and $r'$.

The coefficients $r'_{u, v}$ satisfy a recursion similar to
Theorem~\ref{rrecursion}.

\begin{proposition}
  \label{altrecursion}Suppose that $su > u$. If $v < sv$:
  \begin{equation}
    \label{rirecone} r'_{u, v} = r'_{su, sv} + \frac{q - 1}{1 -
    \mathbf{z}^{u^{- 1} \alpha}} r'_{su, v}
  \end{equation}
  If $u < su$ and $v > sv$:
  \begin{equation}
    \label{rirectwo} r'_{u, v} = \frac{(q - 1) \mathbf{z}^{u^{- 1} \alpha}}{1
    - \mathbf{z}^{u^{- 1} \alpha}} r'_{su, v} + qr'_{su, sv}
  \end{equation}
\end{proposition}

Note that $s u > u$ implies that $u^{- 1} \alpha$ is a positive root.

\begin{proof}
  Since $(r_{u, v}')$ is the inverse matrix of $(r_{u, v})$ we have
  \begin{equation}
    \label{mutvrel} T_{v^{- 1}} = \sum_{u \leqslant v} q_v \overline{r_{w,
    v}'} (\mathbf{z}) \mu_{\mathbf{z}} (w) .
  \end{equation}

  First let us consider the case $v > s v$. Then $T_{(s v)^{- 1}} = T_{v^{-
  1}} T_s^{- 1}$. Moreover for any $w \in W$, we may write $T_s^{- 1}$ as a
  linear combination of $\mu_{w\mathbf{z}} (s)$ and $1$ to obtain
  \begin{equation}
    \label{rdway} T_{(s v)^{- 1}} = \sum_{w \leqslant v} \left( q_v
    \overline{r_{w, v}'} (\mathbf{z}) \mu_{\mathbf{z}} (w) \right) \left(
    \mu_{w\mathbf{z}} (s) - \frac{1 - q^{- 1}}{1 -\mathbf{z}^{w^{- 1}
    \alpha}} \right) .
  \end{equation}
  Then we may use Lemma~\ref{mumuela} to compare the coefficients of
  $\mu_{\mathbf{z}} (s u)$ in this equation and in (\ref{mutvrel}) applied
  to $s v$. In (\ref{rdway}) there are two ways to get a coefficient of $s u$:
  we may either take $w = u$ or $w = s u$. We obtain
  \[ q_{s v} \overline{r'_{s u, s v}} = q_v \overline{r'_{u, v}} -_{}
     \overline{r'_{s u, v}} q_v \cdot \frac{1 - q^{- 1}}{1 -\mathbf{z}^{-
     u^{- 1} \alpha}} . \]
  Applying the involution and rearranging gives (\ref{rirectwo}).
 
  Now let $s v > v$. Then $T_{(s v)^{- 1}} = T_{v^{- 1}} T_s$. We may proceed
  as before except that now it is $T_s$ that we are expressing as a linear
  combination of $\mu_{w\mathbf{z}} (s)$ and $1$. We obtain
  \[ T_{(s v)^{- 1}} = \sum_{w \leqslant v} \left( q_v \overline{r_{w, v}'}
     (\mathbf{z}) \mu_{\mathbf{z}} (w) \right) \left( q
     \mu_{w\mathbf{z}} (s) - \frac{(q - 1) \mathbf{z}^{w^{- 1} \alpha}}{1
     -\mathbf{z}^{w^{- 1} \alpha}} \right) . \]
  Now comparing the coefficient of $\mu_{\mathbf{z}} (s w)$ gives the
  identity
  \[ q_{s v} \overline{r_{s u, s v}'} = q q_v \overline{r'_{u, v}} - q_v
     \overline{r'_{s u, v}} \frac{(q - 1) \mathbf{z}^{- u^{- 1} \alpha}}{1
     -\mathbf{z}^{- u^{- 1} \alpha}} . \]
  Applying the involution and rearranging gives (\ref{rirecone}).
\end{proof}

\begin{proof}[Proof of Theorem~\ref{mmduality}]
  It is sufficient to prove that defining
  \begin{equation}
    \label{rprimeid}
    r'(u,v)=\varepsilon(u)\varepsilon(v)r_{w_0v,w_0u}
  \end{equation}
  makes the recursion of Proposition~\ref{altrecursion}
  true. Since $w\mapsto w_0w$ is a Bruhat-order
  reversing bijection of the Weyl group to intself,
  we may apply Theorem~\ref{rrecursion} with
  $u$, $v$ and $s$ being replaced by $w_0v$,
  $w_0u$ and $w_0sw_0$. With this substitution
  it is easy to see that the definition (\ref{rprimeid})
  makes the recursion (\ref{altrecursion}) true,
  so this definition must agree with our original
  one that makes of $(r'_{u,v})$ being the inverse matrix
  of the matrix $r_{u,v}$. This is equivalent to
  (\ref{first_mm}). To obtain (\ref{second_mm}), we
  substitute equations (\ref{mfromr}) for to
  express $m_{u,v}$ and $m_{w_0v,w_0u}$ into
  the left hand side and use (\ref{first_mm}).
\end{proof}
  
\section{\label{despro}Descent properties of $m_{u,v}$}

Although we will not prove the conjectured formula (\ref{conjecture})
we now have tools to prove it in many cases.

\begin{proposition}
  \label{propertyz}Let $u, v \in W$ and assume that $s$ is a simple reflection
  such that $s u < u$ and $s v < v$. Then the following are equivalent:
  
  (i) $u \leqslant v$,
  
  (ii) $s u \leqslant v$,
  
  (iii) $s u \leqslant s v$.
\end{proposition}

\begin{proof}
  This is Property Z in Deodhar~{\cite{DeodharCharacterizations}}. It is
  sometimes called the \textit{lifting property} of the Bruhat order. See
  \cite{BjornerBrenti} Proposition~2.2.7 for a proof.
\end{proof}

The next result allows computation of $m_{u,v}$ from $m_{u,sv}$ if
a simple reflection $s$ may be found such that $sv<v$ and $su>u$.
If this is true, the map $x\mapsto s x$ is a \textit{special matching}
in the sense of Brenti~\cite{BrentiIntersection} and the reduction is
reminiscent of the proof in certain cases that that the Kazhdan-Lusztig
polynomials are combinatorial invariants of the Bruhat interval poset.  See
See~\cite{BrentiCaselliMarietti}.

\begin{proposition}
  \label{firstdescent}
  Let $u<v$ and let $s=s_\alpha$ be a simple reflection
  such that $s v < v$ and $u < s u$. Then
  \[S (u, v) = S (u, s v) \cup \{ -  v^{- 1} \alpha \}\qquad
  \text{(disjoint)},\]
  and
  \begin{equation}
    \label{mdescent} \overline{m_{u, v}} = \left( \frac{1 - q\mathbf{z}^{-
    v^{- 1} \alpha}}{1 -\mathbf{z}^{- v^{- 1} \alpha}} \right) 
    \overline{m_{u, s v}} .
  \end{equation}
\end{proposition}

\begin{proof}
  Note that by Proposition~\ref{propertyz} we have $u \leqslant s v$. If
  $\beta = - v^{- 1} \alpha$ then $v r_{\beta} = s v$, so $u \leqslant v
  r_{\beta} < v$ is true but $u \leqslant s v r_{\beta} < s v$ is not, showing
  that $- v^{- 1} \alpha \in S (u, v)$ but not $S (u, s v)$. If $\beta$ is a
  positive root not equal to $- v^{- 1} \alpha$ we must show $\beta \in S (u,
  v)$ if and only if $\beta \in S (u, s v)$. First suppose that $s v r_{\beta}
  < v r_{\beta}$. Then this statement is easily deduced from
  Proposition~\ref{propertyz}. Therefore let us assume that $v r_{\beta} < s v
  r_{\beta}$. If $\beta \in S (u, s v)$ then $u \leqslant s v r_{\beta} < s
  v$. Proposition~\ref{propertyz} implies that $u \leqslant v r_{\beta}$, and
  $v r_{\beta} < s v r_{\beta}$ while again by Proposition~\ref{propertyz}, $s
  v r_{\beta} \leqslant v$. Therefore $\beta \in S (u, v)$. We are left to
  check that if $\beta \in S (u, v)$ but $\beta \notin S (u, s v)$ then $\beta
  = - v^{- 1} \alpha$. To do this, we use the Strong Exchange Property for
  Coxeter groups, which is Theorem~5.8 in {\cite{Humphreys}}. Write $v = s_1
  \cdots s_N$ where the $s_i$ are simple reflections, and the expression is
  reduced. Since $s$ is a left descent we may assume that $s_1 = s$. The
  Strong Exchange Property states that $v r_{\beta} = s_1 \cdots \widehat{s_i}
  \cdots s_N$ for some $i$. Suppose that $i \neq 1$. Then $s v r_{\beta} = s_2
  \cdots \widehat{s_i} \cdots s_N < s_2 \cdots s_N = s v$, while by
  Proposition~\ref{propertyz} we have $u \leqslant s v r_{\beta}$. This
  contradicts our assumption that $\beta \notin S (u, s v)$. Therefore $i = 1$
  which implies that $\beta = s_N s_{N - 1} \cdots s_2 (\alpha) = v^{- 1} (-
  \alpha)$.
  
  We turn to (\ref{mdescent}). Using Proposition~\ref{propertyz}, the fact
  that $s v < v$ and $s u < u$ implies that $u \leqslant x \leqslant v$ if and
  only if $u \leqslant s x \leqslant v$. Therefore
  \[ \overline{m_{u, v}} = \sum_{u \leqslant x \leqslant v} r_{x, v} =
     \sum_{\substack{u \leqslant x \leqslant v\\ s x < x}}(r_{x, v} + r_{s x, v}) . \]
  We may now use both cases of Theorem~\ref{rrecursion} to rewrite this. The
  first case of the recursion applies to $r_{x, v}$, and the second applies to
  $r_{s x, v}$. We have
  \[ r_{x, v} + r_{s x, v} = \frac{1 - q}{1 - \mathbf{z}^{- v^{- 1} \alpha}}
     r_{x, sv} (\mathbf{z}) + r_{sx, sv} (\mathbf{z}) + (1 - q) 
     \frac{\mathbf{z}^{- v^{- 1} \alpha}}{1 - \mathbf{z}^{- v^{- 1} \alpha}}
     r_{s x, sv} (\mathbf{z}) + qr_{x, sv} (\mathbf{z}) . \]
  Simplifying, we get
  \[ r_{x, v} + r_{s x, v} = \left( \frac{1 - q\mathbf{z}^{- v^{- 1}
     \alpha}}{1 -\mathbf{z}^{- v^{- 1} \alpha}} \right) (r_{x, s v} + r_{s
     x, s v}) . \]
  The term $r_{x, s v}$ can be zero since it is possible that $x$ is not
  $\leqslant s v$, but we always have $s x \leqslant s v$ by Proposition~\ref{propertyz}.
   Discarding $r_{x, s v}$ when $x$ is not $s v$, we get
  \[ \overline{m_{u, v}} = \left( \frac{1 - q\mathbf{z}^{- v^{- 1}
     \alpha}}{1 -\mathbf{z}^{- v^{- 1} \alpha}} \right) 
     \sum_{\substack{u \leqslant x \leqslant v\\ s x < x}}
     (r_{x, s v} + r_{s x, s v}) = \left( \frac{1 -
     q\mathbf{z}^{- v^{- 1} \alpha}}{1 -\mathbf{z}^{- v^{- 1} \alpha}}
     \right)  \sum_{u \leqslant x \leqslant s v} r_{x, s v} \]
  which equals the right-hand side of (\ref{mdescent}).
\end{proof}

\eject
Here is another type of descent result.

\begin{proposition}
  \label{seconddescent}
  Assume that $s v < v$ and $s u < u$. Assume furthermore that
  $u$ is not $\leqslant s v$.
  \begin{enumerate}
  \item[(i)] Then $S (u, v) = S (s u, s v)$.
    \smallskip
  \item[(ii)] The map $x\mapsto s x$ is a bijection of
    the Bruhat interval $[u,v]=\{x|u\leqslant x\leqslant v\}$
    to $[su,sv]$. If $u\leqslant x\leqslant v$, then
    $s x < x$ and $x$ is not $\leqslant sv$, 
    and $S (x, v) = S (s x, s v)$.
    \smallskip
  \item[(iii)] We have
    \[m_{u,v}=m_{su,sv},\qquad r_{u,v}=r_{su,sv}.\]
  \item[(iv)] If furthermore $Q_{u,v}=1$ then $Q_{s u,s v}=1$.
  \end{enumerate}
\end{proposition}

\begin{proof}
  To prove (i), we first show that $S (u, v) \subseteq S (s u, s v)$. Let
  $\beta \in S (u, v)$ and let $r = r_{\beta}$, so that $u \leqslant v
  r_{\beta}$. Let $v = s_1 \cdots s_N$ be a reduced expression with $s_1 = s$.
  Thus $v r = s_1 \cdots \widehat{s_i} \cdots s_N$ for some $i$. We claim that
  $i \neq 1$. Indeed, if $i = 1$ then $v r = s v$ so $u \leqslant v r$ which
  contradicts our assumption. Therefore $i \neq 1$ and $s v r = s_2 \cdots
  \widehat{s_i} \cdots s_N$. Since $s v = s_2 \cdots s_N$ is a reduced
  expression, we see that $s v r < s v$. On the other hand, since $s u < u$,
  Proposition~\ref{propertyz} implies that $s u < s v r$ and therefore $\beta
  \in S (s u, s v)$.
  
  On the other hand, let us show that $S (s u, s v) \subseteq S (u, v)$. Thus
  assume that $r = r_{\beta}$ where $\beta \in S (s u, s v)$ and $s u
  \leqslant s v r < s v$. We claim that $s v r < v r$. Indeed, if not, then $s
  u \leqslant s v r$ implies $u \leqslant s v r$ by
  Proposition~\ref{propertyz} and so $u \leqslant s v r < s v$, contradicting
  our hypothesis. Now since $v r > s v r$, $u > s u$ and $s u \leqslant s v
  r$, Proposition~\ref{propertyz} implies that $u \leqslant v r$. On the other
  hand since $v > s v$ and $s v r < s v$, Proposition~\ref{propertyz} implies
  that $v r < v$. (We cannot have $v r = v$ since $r$ is a reflection.) Thus
  $u \leqslant v r < v$ and so $\beta \in S (u, v)$. Now (i) is proved.
  
  We prove (ii). First, if $x \in [u, v]$, then we claim that $x > s x$.
  Indeed, if $s x > x$ then $x < s v$ by Proposition~\ref{propertyz}. Then $u
  \leqslant x < s v$, contradicting our hypothesis. Now two applications of
  Proposition~\ref{propertyz} show that $s u \leqslant s x$ and $s x \leqslant
  s v$. Thus $x \mapsto s x$ maps $[u, v]$ into $[s u, s v]$. The fact that
  this map is surjective also follows from Proposition~\ref{propertyz}.
  Finally, since we have shown that $x < s x$ for $x \in [u, v]$, part (i)
  applies to the pair $x, v$, implying that $S (x, v) = S (s x, s v)$. Now
  (ii) is proved.
  
  As for (iii), the fact that $r_{u,v}=r_{su,sv}$ follows from Theorem~\ref{rrecursion}
  since $r_{u,sv}=0$ under our assumption that $u$ is not $\leqslant sv$.
  By (ii), we have similarly $r_{x,v}=r_{sx,sv}$ for $x\in[u,v]$.
  Summing over $x$ and applying the involution gives $m_{u,v}=m_{su,sv}$.

  We prove (iv). A criterion for $P_{x, y} = 1$ due to Kazhdan and
  Lusztig~{\cite{KazhdanLusztig}} Lemma 2.6 is that $\sum_{x \leqslant z
  \leqslant y} R_{x, z} = q_y q_x^{- 1}$. (Actually in this Lemma this is the
  condition that $P_{z, y} = 1$ for all $x \leqslant z \leqslant y$, but by a
  result of Carrell and Peterson~\cite{Carrell}, that is equivalent to $P_{x, y} = 1$.)
  By~{\cite{KazhdanLusztig}} Lemma~2.1 (iv) it follows that the criterion for
  $Q_{x, y} = 1$ is that $\sum_{x \leqslant z \leqslant y} R_{z, y} = q_y
  q_x^{- 1}$. Thus $Q_{u, v} = 1$ we have $\sum_{u \leqslant x \leqslant v}
  R_{x, v} = q_v q_u^{- 1}$. Moreover using (ii) and the recurrence
  {\cite{KazhdanLusztig}} (2.0.b) for $R$ we have $R_{x, v} = R_{s x, s v}$
  and it follows that $\sum_{s u \leqslant s x \leqslant s v} R_{s x, s v} =
  q_v q_u^{- 1} = q_{s v} q_{s u}^{- 1}$. Therefore $Q_{s u, s v} = 1$.
\end{proof}

We make the following conjecture.

\begin{conjecture}
  \label{descentconjecture}
  Assume that $\Phi$ is simply-laced and that $u<v$ in $W$ such
  that $Q_{u,v}=1$. Then there exists a simple reflection $s\in W$
  such that either:
  \begin{enumerate}
    \item[(i)] $s v<v$ and $s u> u$; or
    \item[(ii)] $s v<v$ and $s u< u$, and $u$ is not $\leqslant sv$.
  \end{enumerate}
\end{conjecture}

We have checked this (using Sage) for Cartan types $A_5$ and $D_4$.
For $A_5$ we find 1346 pairs $u<v$ such that no descent $s$ of $v$
exists satisfying either (i) or (ii), and for each of these, we have
$Q_{u,v}\neq 1$. For example, we can take $(u,v)=(s_2,s_2s_1s_3s_2)$
and the only descent $s=s_2$ of $v$ does not satisfy either (i) or (ii),
but this does not contradict the conjecture since $Q_{u,v}=1+q$.

\begin{theorem}
  Assume that $\Phi$ is simply-laced and $Q_{u,v}=1$.
  Then Conjecture~\ref{descentconjecture} implies
  (\ref{conjecture}), which is conjectured in~\cite{BumpNakasuji}
  under these assumptions.
\end{theorem}

\begin{proof}
  We assume that $u<v$ and $Q_{u,v}=1$. Choose a left descent $s$ of $v$.
  If $su>u$, then Proposition~\ref{firstdescent} applies. Note
  \[m_{u, sv} = m_{u, sv} (\mathbf{z}) = \prod_{\beta \in S
    (u, sv)} \frac{1 - q^{- 1} \mathbf{z}^{\beta}}{1 -\mathbf{z}^{\beta}} .
  \]
  Now Proposition~\ref{firstdescent} implies (\ref{conjecture})
  for $u,v$.

  On the other hand if $s u<u$, then on Conjecture~\ref{descentconjecture}
  we have $u$ not $\leqslant sv$ and so Proposition~\ref{seconddescent}
  applies, and again the result follows.
\end{proof}

We end with another puzzle. We give the root lattice the usual partial order
in which $x \geqslant y$ if $x - y$ lies in the cone generated by the positive
roots. Then the set $T$ of reflections has a partial order in which if
$\alpha, \beta \in \Phi^+$ then $r_{\alpha} \geqslant r_{\beta}$ if and only
if $\alpha \geqslant b$. Let $\operatorname{AD} (u, v) = \{ r \in T|r u > u, r v < v \}$. We
will write $u \vartriangleleft v$ to denote the covering relation in the
Bruhat order. Thus $u \vartriangleleft v$ if $u < v$ and $l (u) = l (v) - 1$.

\begin{theorem}[Tsukerman and Williams~{\cite{TsukermanWilliams}}, Caselli and
Sentinelli~{\cite{CaselliSentinelli}}]
  Suppose that $\Phi$ is a simply-laced root system. Suppose $u < v$. Then
  $\operatorname{AD} (u, v)$ is nonempty and if \ $t$ is a minimal element, then $u
  \vartriangleleft t u \leqslant v$ and $u \leqslant t v \vartriangleleft u$.
  In this case
  \[ R_{u, v} = q R_{t u, t v} + (q - 1) R_{u, t v} . \]
\end{theorem}

Suppose in the setting of this theorem that $t = r_{\alpha}$
($\alpha\in\Phi^+$). Let $\beta =
- v^{- 1} \alpha$. Then $\beta$ is a positive root. To generalize
Theorem~\ref{rrecursion} it is natural to ask whether
\begin{equation}
 \label{adrecursion}
 r_{u, v} = q r_{s u, s v} + (q - 1)
   \frac{\mathbf{z}^{\beta}}{\mathbf{z}^{\beta} - 1} r_{u, s v}, \qquad
   \beta = - v^{- 1} \alpha .
\end{equation}
This is often but not always true. For $A_3$, it fails in the following cases:
\[ \begin{array}{|l|l|l|l|l|}
     \hline
     u & v & t & P_{u, v} & Q_{u, v}\\
     \hline
     s_1 & s_1 s_2 s_3 s_2 s_1 & s_1 s_2 s_1 & 1 + q & 1\\
     \hline
     s_3 & s_1 s_2 s_3 s_2 s_1 & s_2 s_3 s_2 & 1 + q & 1\\
     \hline
     s_1 s_3 & s_1 s_2 s_3 s_2 s_1 & s_1 s_2 s_1 & 1 + q & 1 + q\\
     \hline
     s_1 s_3 & s_1 s_2 s_3 s_2 s_1 & s_2 s_3 s_2 & 1 + q & 1 + q\\
     \hline
     s_2 & s_2 s_1 s_3 s_2 & s_1 s_2 s_1 & 1 + q & 1 + q\\
     \hline
     s_2 & s_2 s_1 s_3 s_2 & s_2 s_3 s_2 & 1 + q & 1 + q\\
     \hline
     s_2 & s_3 s_2 s_1 s_3 s_2 & s_1 s_2 s_1 & 1 & 1 + q\\
     \hline
     s_2 & s_1 s_2 s_1 s_3 s_2 & s_2 s_3 s_2 & 1 & 1 + q\\
     \hline
\end{array} \]
Except in these cases, we have not only (\ref{adrecursion})
but also
\begin{equation}
\label{admrec}
\overline{m_{u,v}}=\left(\frac{1-q\mathbf{z}^\beta}{1-\mathbf{z}^\beta}\right)
\overline{m_{tu,tv}}.
\end{equation}
We may conjecture that (\ref{adrecursion}) and (\ref{admrec}) are
true if both $P_{u,v}=Q_{u,v}=1$. (This has been checked for $A_4$
as well as $A_3$.) This does not imply the conjecture
(\ref{conjecture}) because of the unwanted requirement $P_{u,v}=1$,
and so Conjecture~\ref{descentconjecture} seems a better approach.

\bibliographystyle{hplain}
\bibliography{casselman}

\end{document}